\documentclass[sn-mathphys-num]{sn-jnl}

\usepackage{graphicx}%
\usepackage{multirow}%
\usepackage{amsmath,amssymb,amsfonts}%
\usepackage{amsthm}%
\usepackage{mathrsfs}%
\usepackage[title]{appendix}%
\usepackage{xcolor}%
\usepackage{textcomp}%
\usepackage{manyfoot}%
\usepackage{booktabs}%
\usepackage{algorithm}%
\usepackage{algorithmicx}%
\usepackage{algpseudocode}%
\usepackage{listings}%
\usepackage{enumitem}
\newcommand{\bal}[1]{\begin{align*}#1\end{align*}}

\newcommand{\Z}{\mathbb{Z}} %ring of integers
\newcommand{\N}{\mathbb{N}} %natural numbers
\newcommand{\R}{\mathbb{R}}
\newcommand{\bC}{\mathbb{C}}
\newcommand{\C}{\mathbb{C}}
\newcommand{\E}{\mathbb{E}}

\renewcommand{\P}{\mathbb{P}}

\newcommand{\sym}{\mathrm{sym}}

\newcommand{\eps}{\varepsilon}

\renewcommand{\AA}{\mathcal A}
\newcommand{\BB}{\mathcal B}
\newcommand{\MM}{\mathcal M}
\newcommand{\VV}{\mathcal V}

\DeclareMathOperator{\Log}{Log}
\DeclareMathOperator{\Arg}{Arg}

\theoremstyle{plain}
\newtheorem{theo}{Theorem}[section]
\newtheorem{lemma}[theo]{Lemma}
\newtheorem{propo}[theo]{Proposition}
\newtheorem{coro}[theo]{Corollary}

\theoremstyle{definition}
\newtheorem{defi}[theo]{Definition}

\theoremstyle{remark}
\newtheorem{remark}[theo]{Remark}

\newtheorem*{solu*}{Solution}

\begin{document}

\title[Free compressions of R-diagonal variables]{Free compressions of R-diagonal random variables and the semigroup of Brown measures}


\author*[1]{\fnm{Vladislav} \sur{Kargin}}\email{vkargin@binghamton.edu}


\affil*[1]{\orgdiv{Department of Mathematics and Statistics}, \orgname{Binghamton University}, \orgaddress{\street{4400 Vestal Pkwy E}, \city{Binghamton}, \postcode{13902}, \state{NY}, \country{USA}}}

%\affil[2]{\orgdiv{Department}, \orgname{Organization}, \orgaddress{\street{Street}, \city{City}, \postcode{10587}, \state{State}, \country{Country}}}
%
%\affil[3]{\orgdiv{Department}, \orgname{Organization}, \orgaddress{\street{Street}, \city{City}, \postcode{610101}, \state{State}, \country{Country}}}

\abstract{We investigate the Brown measures of compressions of \( R \)-diagonal random variables, extending previous results to include unbounded cases. For random variables with finite variance, we demonstrate that the Brown measures of their compressions converge to the uniform distribution on the unit disc. In the case of infinite variance, we characterize the Brown measures that remain stable under the compression operation and explore their properties in detail.}


\maketitle

\section{Introduction}

We examine the behavior of $R$-diagonal free random variables under compression. For bounded $R$-diagonal random variables, this problem was previously studied by \citeauthor{cor2023} in \cite{cor2023}, particularly in the context of applying free probability theory to the zeros of multiply-differentiated polynomials. Here, we extend these results to the unbounded case. 

Additional motivation for this work arises from random matrix theory, where truncations of Haar-distributed random unitary matrices have been extensively investigated (\cite{diaconis2003}, \cite{jiang2006}, \cite{zs2000}, \cite{petz_reffy2005}, \cite{dong_jiang_li2012}). Broadly speaking, these studies reveal that an $m \times m $ truncation of an $ n \times n $ Haar matrix converges to a matrix with i.i.d. Gaussian entries, both in the distribution of its entries and the distribution of its eigenvalues, provided $m$ is small relative to $n$—for example, when $m = o(n)$. When $m/n \to c $ as both $ n $ and $m$ grow, the limiting distribution of eigenvalues has been analyzed in \cite{zs2000}, \cite{petz_reffy2005}, and \cite{dong_jiang_li2012}.

Our first result is as follows. 
\begin{theo}
\label{theo_convergence_compressions}
Let $x \in \AA^\Delta$ be an $R$-diagonal non-commutative random variable, and assume that $\varphi(|x|^2) = 1$. Assume that $P_s$ is a projection with normalized trace $t \equiv 1/s \in (0, 1]$, which is free of $x$.  Then, as $s \to \infty$,  the Brown measure of free compression $\sqrt{s} \pi_s(x) = \sqrt{s} [P_s x P_s]\in \widetilde \AA_P$ weakly converges to the uniform measure on the unit disc.
\end{theo}

We introduce most of the notation later; for now, note that $[P_s x P_s]$ serves as an analogue of matrix truncation with  $n/m = s > 1$. The $R$-diagonal non-commutative random variable  $x$ belongs to a class that generalizes unitarily-invariant random matrices. Notably, this class includes Haar unitaries, which are infinite-dimensional analogues of Haar-distributed random matrices. The Brown measures of these variables act as analogues of empirical eigenvalue distributions for matrices.

The key innovation in this result is that we do not require the non-commutative random variable $x$ to be bounded; we only assume that its second moment, $\varphi(|x|^2)$, is finite.

Relaxing the finiteness assumption on the second moment allows for a broader range of behaviors. First, let us define a semigroup on Brown measures of $R$-diagonal elements. 

%------------------------------Definition of the semigroup----------------

Recall that for a probability measure $\nu$ supported on the real line the free additive convolution semigroup $\nu^{\boxplus s}$ was defined in Theorem 2.5 in \cite{belinschi_bercovici04}.

\begin{defi}
\label{defi_semigroup}
Let $\mu \in \MM_{BR}$ be the Brown measure of an $R$-diagonal random variable $x = uh \in \AA^\Delta$ and let $s \geq 1$. Let $\nu =  \mu_{h}^\sym = \mu_{|x|}^\sym$ denote the symmetrized measure of $h$.  Let $b_s$ be an even self-adjoint random variable which is $\ast$- free of $u$ and has measure $\nu^{\boxplus s}$, and let $x_s = u b_s$. We define $\mu^{\boxplus s}$ as the Brown measure of $x_s$. 
\end{defi}

\begin{propo}
\label{lemma_semigroup}
Let  $\mu \neq \delta_0$ be the Brown measure of a (possibly unbounded) $R$-diagonal element $x\in \AA^{\Delta}$. Then $\mu^{\boxplus s}$ exists for all $s \geq 1$ and $\{\mu^{\boxplus s}\}_{s\ge 1}$ forms
a semigroup  of probability measures on $\mathbb C$ such that
\begin{enumerate}[label = (\Alph*), ref = \Alph*]
\item \label{semigroup_at_1} $\mu^{\boxplus 1}=\mu$;

\item  \label{semigroup_additivity} for $s_1,s_2\ge 1$, $\mu^{\boxplus (s_1+s_2)}$ is the Brown measure of $x_{s_1}+x_{s_2}\in \AA^{\Delta}$ where
$x_{s_1}$ and $x_{s_2}$ are $\ast$-free $R$-diagonal elements in $\AA^{\Delta}$ whose Brown measures are
$\mu^{\boxplus s_1}$ and $\mu^{\boxplus s_2}$, respectively;

\item \label{semigroup_continuity} the map $[1, \infty) \ni s \mapsto \mu^{\boxplus s}$ is continuous with respect to weak convergence of probability measures on $\bC$.
\end{enumerate}
We refer to $\{\mu^{\boxplus s}\}_{s\ge 1}$ as the semigroup generated by $\mu$.
\end{propo}

A proof of this result will be given in Section \ref{section_results}.

As we show in Theorem \ref{theo_semigroups}, this semigroup is realized by free compressions: $\mu^{\boxplus s}$ equals the Brown measure of $s\pi_s(x) = s[P_s x P_s]$, where $P_s$ is a free projection with trace $1/s$.

We say that a Brown measure $\mu \ne \delta_0 \in \MM_{BR}$  is \emph{stable} if 
\bal{
(\mu^{\boxplus k}) (A) = \mu(\alpha_k A),
}  
with $\alpha_k > 0$, for all $k \in \N$   and all Borel sets $A \subset \bC$.

\begin{defi}[Measures $\mu_\beta$]
\label{defi_mu_beta}
For $\beta\ge 0$, define the probability measure $\mu_\beta$ on $\C$ by requiring that it is
rotationally invariant and that its radial distribution function satisfies
\begin{equation}\label{eq:mu_beta_quantile}
\mu_\beta\!\left[B\!\left(\frac{t^{1/2}}{(1-t)^{\beta/2}}\right)\right]=t,
\qquad t\in(0,1),
\end{equation}
where $B(r)=\{z\in\C:\ |z|<r\}$.
\end{defi}

For $c>0$ and a probability measure $\mu$ on $\C$,  we denote by $D_c\mu$ the dilation of this measure, that is
$
(D_c\mu)(A):=\mu(c A),$ for all Borel  $A\subset\C$.

\begin{theo}
\label{theo_stable_main}
Let $\mu$ be the Brown measure of an $R$-diagonal element $x\in \AA^{\Delta}$, and let $\{\mu^{\boxplus s}\}_{s\ge 1}$ be the semigroup
generated by $\mu$. Then $\mu$ is a stable Brown measure if and only if $\mu=D_c(\mu_\beta)$ for some $c>0$ and $\beta\ge 0$.
\end{theo}

The stable measures for sums of free identically distributed $R$-diagonal random variables $x_i$ were investigated by \citeauthor{kosters_tikhomirov2018} in \cite{kosters_tikhomirov2018} within the context of random matrix theory (see Definition 3.1 in their paper). Our contribution is twofold: first, we provide a more explicit characterization of these measures; second, we demonstrate that the stability of Brown measures under additive convolutions also extends to stability under the compression operation.

Additional properties of stable Brown measures are discussed in Section \ref{section_stable_measures}.

The proof of our main results connects compressions of a random variable  $x$ to the additive convolution semigroup of Brown measures, where the results are either already known or easier to establish. Since we do not assume the boundedness of random variables, we cannot use free cumulants to demonstrate that compressions satisfy the semigroup property. Instead, we define the free semigroup of Brown measures directly and verify that a compression of a random variable corresponds to a measure within this semigroup. Consequently, rather than relying on free cumulants, our approach heavily employs the analytic techniques developed by Haagerup and collaborators in \cite{haagerup_larsen00}, \cite{haagerup_schultz2007}, and \cite{haagerup_moller2013}.

The main idea behind this construction is as follows: for a non-selfadjoint random variable, we define its singular-value measure, which is supported on $\mathbb{R}$. Using results by Belinschi and Bercovici in \cite{belinschi_bercovici04}, we define a semigroup for singular-value measures and realize this semigroup in terms of random variables by compressing a self-adjoint random variable. To construct the semigroup of Brown measures, we multiply these compressed self-adjoint random variables by free Haar unitaries. The primary task is to verify that the Brown measures of these resulting products satisfy the semigroup property. The detailed arguments are deferred to the next section.

For the semigroup defined in this way we can prove the following result. Recall that any $R$-diagonal random variable can be decomposed as $x = u h$, where $u$ is a Haar unitary, $h$ is positive semidefinite, and $u$ and $h$ are free. 

 %%%%%%%%%%%%%%%%%%%%%%%%%
%Theo_semigroups Unbounded
%%%%%%%%%%%%%%%%%%%%%%%%%%
 \begin{theo}
 \label{theo_semigroup_unbounded}
 Let $x = u h \in \AA^\Delta$ be an $R$-diagonal random variable with Brown measure $\mu$, and let $x_s = u h_s$, $s \geq 1$, be $R$-diagonal random variables with Brown measures $\mu^{\boxplus s}$.
 Then,
\begin{equation}
\label{semigroup_S}
 S_{h_s^2}(z) = \frac{1}{s}\frac{1 + z/s}{1 + z} S_{h^2}\Big(\frac{z}{s}\Big). 
 \end{equation}
 \end{theo}
 In conjunction with the results of \citeauthor{haagerup_larsen00} in \cite{haagerup_larsen00}, and \citeauthor{haagerup_schultz2007} in \cite{haagerup_schultz2007}, formula (\ref{semigroup_S}) provides a method to calculate the distribution of measure $\mu^{\boxplus s}$ for every $s \geq 1$. 

 Formula \eqref{semigroup_S} was discovered by \citeauthor{cor2023} in \cite{cor2023}  and proved for bounded $R$-diagonal random variables by using free cumulants techniques. We extend it to unbounded random variables. 
 
  The next theorem shows that free compressions of an $R$-diagonal random variable that has the Brown measure $\mu$ form the semigroup $\mu^{\boxplus s}$.  
%%%%%%%%%%%%%%%%%%%%%%%%%
%Theo_semigroups
%%%%%%%%%%%%%%%%%%%%%%%%%%
%\begin{theo}
%\label{theo_semigroups}
%Let $\mu$ be the Brown measure of an $R$-diagonal variable $x = uh \in \AA^\Delta$ and let $(\mu^{\boxplus s})_ {s \geq 1}$ be the semigroup of Brown measures generated by $\mu$. Then $\mu^{\boxplus s}$ coincide with the Brown measure of $s\pi_s(x) = s [P_s x P_s]\in \tilde \AA_P$, where $P_s \in \AA$ is a projection which is $\ast$-free from $(u, h)$, with $\varphi(P_s) = \frac{1}{s}$. 
%\end{theo}

\begin{theo}
\label{theo_semigroups}
Let $\mu$ be the Brown measure of an $R$-diagonal variable $x \in \AA^\Delta$ and let $(\mu^{\boxplus s})_ {s \geq 1}$ be the semigroup of Brown measures generated by $\mu$. Then $\mu^{\boxplus s}$ coincides with the Brown measure of $s\pi_s(x) = s [P_s x P_s]\in \widetilde \AA_{P_s}$, where $P_s \in \AA$ is a projection which is $\ast$-free from $x$, with $\varphi(P_s) = \frac{1}{s}$. 
\end{theo}

Theorems \ref{theo_semigroup_unbounded} and \ref{theo_semigroups} establish the connection between free compressions and free semigroups of $R$-diagonal random variables, serving as the primary tools for proving Theorems \ref{theo_convergence_compressions} and \ref{theo_stable_main}.

The rest of the paper is organized as follows: Section \ref{section_preliminaries} introduces the notation and reviews the necessary theoretical background. Section \ref{section_results} presents the proofs of our main results and explores additional properties of stable Brown measures.

%%%%%%%%%%%%%%%%%%%%%%%%%%%%%%%
%%%%%%%%%%%%%%%%%%%%%%%%%%%%%%%
%%%%%%%%%%%%%%%%%%%%%%%%%%%%%%%

\section{Preliminaries and Notation}
\label{section_preliminaries}

\subsubsection*{Basic notions}
%Let $(\AA, \varphi)$ be a \emph{non-commutative tracial probability space}. By this we mean that $\AA$ is a complex associative unital $\ast$-algebra and $\varphi$ is a linear functional $\AA \to \bC$ with $\varphi(e) = 1$. It is assumed that $\varphi$ is a tracial state so that $\varphi(a^\ast a) \geq 0$ for all $a \in \AA$ and $\varphi(a b) = \varphi(b a)$ for all $a$ and $b$ in $\AA$. 

Let $(\AA,\varphi)$ be a tracial $W^\ast$-probability space, i.e., $\AA$ is a (unital) von Neumann algebra and $\varphi$ is a faithful normal tracial state, and let $\widetilde{\AA}$ denote the algebra of $\varphi$-measurable operators affiliated with $\AA$. The elements of $\widetilde \AA$ are called \emph{non-commutative random variables}. 

For free probability, the central concept is freeness of several random variables, which is an analogue of independence in classical probability theory.  (\cite{voiculescu83}, \cite{voiculescu_dykema_nica92}, \cite{nica_speicher06}). Two unital subalgebras $\AA_1$ and $\AA_2$ of $\AA$ are called free, if
\bal{
\varphi(a_1 \ldots a_k) = 0, 
}
whenever three conditions are satisfied: (1) $a_j \in A_{i(j)}$,  $i(j)\in \{1, 2\}$ for all $j = 1,\ldots, k$; (2) $\varphi(a_j) = 0$ for all $j = 1, \ldots, k$, and (3) neighboring elements are from different subalgebras, i.e., $i(1) \ne i(2)$, $i(2) \ne i(3)$, \ldots, $i(k - 1) \ne i(k)$. 

\begin{remark}[Convention: freeness of operators]
Freeness is a notion for unital subalgebras. When we say that (possibly unbounded)
operators $x_1,\dots,x_n$ affiliated with $(A,\varphi)$ are \emph{free}, we mean that
the von Neumann subalgebras $W^\ast(x_i)\subset A$ are free, where
$W^\ast(x_i)$ denotes the von Neumann algebra generated by the bounded functional calculus
of $x_i$ (equivalently, by its spectral projections; or by resolvents when convenient).
The phrase ``$\ast$-free'' is used synonymously with ``free'' and only emphasizes that adjoints are included.
\end{remark}

%%%%%%%%%%%%%%%%%%%%%%%%%%%%%%%%
%%
%%%%%%%%%%%%%%%%%%%%%%%%%%%%%%%%%%
%
%\subsection{Free Convolutions}
%
\subsubsection*{Free additive convolution}
Let $a\in\widetilde{\AA}$ be self-adjoint and let $E_a(\,\cdot\,)$ be the spectral projection-valued measure of $a$, so that
$a=\int_{\mathbb R} t\, dE_a(t)$.

The (spectral) distribution of $a$ is the Borel probability measure $\mu_a$ on $\mathbb R$ defined by
$\mu_a(B):=\varphi\!\big(E_a(B)\big)$, where  $B\subset\mathbb R$ is  Borel. Equivalently, for every bounded Borel function $f:\mathbb R\to\mathbb C$,
\[
\int_{\mathbb R} f(t)\, d\mu_a(t)=\varphi\!\big(f(a)\big),
\]
where $f(a)=\int_{\mathbb R} f(t)\, dE_a(t)$ is defined by the (bounded) functional calculus.
In particular, if $a\in\AA$ is bounded then $\mu_a$ is supported on $[-\|a\|,\|a\|]$, and for every $n\ge0$,
\[
\varphi(a^n)=\int_{\mathbb R} t^n\, d\mu_a(t).
\]
(For unbounded $a\in\widetilde{\AA}$, the identity $\varphi(a^n)=\int t^n\, d\mu_a(t)$ holds whenever
the moment is finite.)

 If $a$ and $b$ are two free self-adjoint random variables, then the measure of $a + b$ is called the \emph{free additive convolution} of measures $\mu_a$ and $\mu_b$ and it is denoted $\mu_a \boxplus \mu_b$. The existence of $\mu_a \boxplus \mu_b$ for unbounded measures was established in (\cite{bercovici_voiculescu93}). 

The free convolution of measures can be characterized in terms of $R$-transforms.  
Let $\mu$ be a probability measure on $\mathbb R$ and
\[
G_\mu(z):=\int_{\mathbb R}\frac{1}{z-t}\,d\mu(t),\qquad z\in\mathbb C\setminus\mathbb R,
\]
its Cauchy transform. It is standard that $G_\mu(\mathbb C^+)\subset\mathbb C^-$ and
$G_\mu(z)\sim 1/z$ as $|z|\to\infty$ nontangentially. In particular (see, e.g., \cite[Section~5]{bercovici_voiculescu93}),
there exist $\alpha,\beta>0$ such that $G_\mu$ is univalent on the truncated cone
\[
\Gamma_{\alpha,\beta}:=\{z\in\mathbb C^+:\ \Im z>\alpha|\Re z|,\ |z|>\beta\}.
\]
We denote by $G_\mu^{(-1)}$ the analytic inverse of $G_\mu$ defined on the domain
$G_\mu(\Gamma_{\alpha,\beta})\subset\mathbb C^-$; this domain contains an interval 
$i(-\varepsilon,0)$ for some $\varepsilon>0$, see \cite[Proposition 5.4]{bercovici_voiculescu93}. 
The $R$-transform is then
\[
R_\mu(w):=G_\mu^{(-1)}(w)-\frac1w,\qquad w\in G_\mu(\Gamma_{\alpha,\beta}).
\]

We have $R_{\mu_a \boxplus \mu_b}(z) = R_{\mu_a}(z) + R_{\mu_b}(z)$. (All identities involving $R_\mu$
 are understood for $z$ in the common domain near $0$ where both sides are defined.)

In \cite{nica_speicher96} and \cite{belinschi_bercovici04}, a free semigroup  $\mu_t \equiv \mu^{\boxplus t}$ was defined satisfying $R_{\mu_t} (z) = t R_\mu(z)$. (See Theorem 2.5 in \cite{belinschi_bercovici04}). 

%%%%%%%%%%%%%%%%%%%%%%%%%%%%%%%%%%
%%Free compressions
%%%%%%%%%%%%%%%%%%%%%%%%%%%%%%%%
%
%\subsection{Free Compressions}
%\begin{defi}
%Let $P$ be a self-adjoint projection such that $\varphi(P) = \frac{1}{s}$, where $s \geq 1$.   
%The \emph{compressed probability space} is a pair $(\AA_P,  \tilde\varphi)$, where $\AA_P$ consists of the elements of the form $[P a P]$, $a \in \AA$ and $\tilde\varphi([P a P]) := s \varphi(a)$.
%\end{defi}
% (The brackets $[\cdot]$ are useful to distinguish the algebraic structures of $\AA_P$ from that of $A$. For example,  $[P]$ is a unit in $\AA_P$ although $P$ is not a unit in $\AA$.  Note that $\tilde\varphi([P]) = 1$.)

\begin{defi}
Let $P\in\AA$ be a self-adjoint projection with $\varphi(P)=\frac{1}{s}$, where $s\ge1$.
The \emph{compressed probability space} is the pair $(\AA_{P},\tilde\varphi)$, where
$\AA_{P}:=\{[PaP]:a\in\AA\},$ 
with operations $[PaP]\cdot[PbP]:=[PaPbP]$ and $[PaP]^*:=[Pa^*P]$ (so that $[P]$ is the unit of $\AA_P$),
and $\tilde\varphi([PaP]) := s\,\varphi(PaP)$, for $a\in\AA.$
\end{defi}

(The brackets $[\cdot]$ distinguish the algebraic structure of $\AA_P$ from that of $\AA$:
for instance, $[P]$ is the unit in $\AA_P$ although $P$ is not the unit in $\AA$.
Note that $\tilde\varphi([P])=s\varphi(P)=1$.)

The operation $\pi: a \to  [P a P]$ is called a \emph{compression}. It is an analogue of the truncation operator for matrices. The free additive semigroup of measures on $\R$ can be realized in terms of free compressions. 
Namely, let $\mu_a$ be the measure of a self-adjoint random variable $a\in \widetilde \AA$. Then $[P_s (s a) P_s] \in  \widetilde \AA_P$ has the measure $\mu_a^{\boxplus s}$. For the proof see Corollary 1.14 of   \cite{nica_speicher96}, where the statement is proved for bounded random variables.  It can be extended to unbounded $a \in \widetilde \AA$ by approximating $a$ with truncations and using the continuity of $R_\mu$ with respect to $\mu$ (Proposition 5.7 in  \cite{bercovici_voiculescu93}).

   %%%%%%%%%%%%%%%%%%%%%%%%%%%%
\subsubsection*{Free multiplicative convolution}
If $a\geq 0$ and $b\geq 0$ are two free self-adjoint random variables that have measures $\mu_a$ and $\mu_b$ supported on $\R^+ = [0, \infty)$, then the random variable $a^{1/2} b a^{1/2}$ has the same distribution as $b^{1/2} a b^{1/2}$ and this distribution  is called the \emph{free multiplicative convolution} of measures $\mu_a$ and $\mu_b$, denoted $\mu_a \boxtimes \mu_b$. This convolution is linearized by the \emph{$S$-transform}, which is defined as follows.

Let $\mu \in \mathrm{Prob}\big(\R^+, \BB\big)$ and $\mu \ne \delta_0$. Define $\psi_\mu: \bC \backslash [0, \infty) \to \bC$ by 
\begin{equation*}
\psi_{\mu}(z) = \int_0^\infty \frac{zt}{1 - zt} \, d \mu(t)  = \frac{1}{z}G_\mu\Big(\frac{1}{z}\Big) - 1. 
\end{equation*}
By \cite[Proposition~1]{arizmendi_perez-abreu2009}, $\psi_\mu$ is univalent on the left half-plane $i\bC^+=\{w\in\bC:\Re w<0\}$.
We set
$\VV_\mu:=\psi_\mu(i\bC^+),$
and denote by $\chi_\mu:\VV_\mu\to i\bC^+$ the analytic inverse of $\psi_\mu$. The $S$-transform is defined by
\[
S_\mu(z) := \frac{1 + z}{z}\,\chi_\mu(z), \qquad z \in \VV_\mu.
\]
(In particular, $\VV_\mu\cap\R=(\mu(\{0\})-1,0)$, so $(-\varepsilon,0)\subset \VV_\mu$ for some $\varepsilon>0$.)
Then, by Corollary~6.6 in \cite{bercovici_voiculescu93}, 
$
S_{\mu_a \boxtimes \mu_b}(z) = S_{\mu_a}(z)\, S_{\mu_b}(z),
$
for $z$ in the connected component of $\VV_{\mu_a}\cap \VV_{\mu_b}$ containing $(-\varepsilon,0)$.

For symmetric measures $\mu$ on $\R$, following \cite{arizmendi_perez-abreu2009}, $\psi_\mu$ is univalent on each of the sectors
\[
H:=\{z\in\bC^+:\ |\Re z|<\Im z\},\qquad 
\widetilde H:=\{z\in\bC^-:\ |\Re z|<|\Im z|\},
\]
yielding two corresponding branches of $\chi_\mu$ (and hence of $S_\mu$) with domains $\psi_\mu(H)$ and $\psi_\mu(\widetilde H)$.

Accordingly one defines two $S$-transforms by
\[
S_\mu(z):=\frac{1+z}{z}\,\chi_\mu(z)\quad (z\in\psi_\mu(H)),\qquad
\widetilde S_\mu(z):=\frac{1+z}{z}\,\widetilde\chi_\mu(z)\quad (z\in\psi_\mu(\widetilde H)).
\]
If $\nu\in \mathrm{Prob}\big(\R^+, \BB\big) $ and $\nu\neq\delta_0$, then $\mu\boxtimes\nu$ is symmetric and
\[
S_{\mu\boxtimes\nu}(z)=S_\mu(z)S_\nu(z),\qquad 
\widetilde S_{\mu\boxtimes\nu}(z)=\widetilde S_\mu(z)S_\nu(z),
\]
on the common domain containing $(-\varepsilon,0)$ for small $\varepsilon>0$ (\cite[Theorem~7]{arizmendi_perez-abreu2009}).

We write $S_\mu$ for both the usual  $S$-transform for $\mu \in \mathrm{Prob}([0, \infty), \BB)$ and for the symmetric $S$-transform when $\mu$ is symmetric.

We will make use of the following well-known formula. Let $\mu \in \mathrm{Prob}\big([0, \infty), \BB\big)$ and assume that $\mu \ne \delta_0$. 
\begin{equation}
\label{equ_S_HL}
S_{\mu}(z) = \frac{1}{z} \Big(z R_\mu(z)\Big)^{(-1)}, \text{ for } z \in (-\eps, 0),
\end{equation}
 where superscript $(-1)$ denotes the functional inverse. 
 
 For compactly supported measures, \eqref{equ_S_HL} follows from the proof of formula (2.5) in \cite{haagerup_larsen00}. Haagerup–Larsen remark that the same argument extends to unbounded measures; for the reader’s convenience, we provide a self-contained proof in Appendix~\ref{section_proof_HL_formula}. In addition, \eqref{equ_S_HL} remains valid for symmetric measures on $\R$, when $S_{\mu}$ is appropriately defined; see Proposition~9 in \cite{arizmendi_perez-abreu2009}.

The following scaling relations are useful in calculations. If $t > 0$, and $a$ is a self-adjoint variable in $\AA$, then 
$R_{ta}(z) = t R_{a}(t z)$ and
$S_{ta}(z) = \frac{1}{t}S_a(z)$.

%%%%%%%%%%%%%%%%%%%%%%%%%%%%%%
%$R$-diagonal Random Variables
%%%%%%%%%%%%%%%%%%%%%%%%%%%%%%
\subsubsection*{$R$-diagonal random variables}
An element $u \in \AA$ is called a \emph{Haar unitary}, if it is unitary ($u^\ast = u^{-1}$) and $\varphi(u^n) = \delta_{n,0}$ for all $n \in \Z$. 
\begin{defi}
\label{defi_even}
 A self-adjoint element $b \in \widetilde \AA$  is \emph{even} if it has \emph{symmetric} distribution: $\mu_b(A)=\mu_b(-A)$ for all
Borel $A\subset\mathbb R$.
\end{defi}
 
% \begin{defi}
%An element $x\in\widetilde\AA$ is called \emph{$R$-diagonal} if it can be written as
%$x=u b$, where $u\in\AA$ is a Haar unitary and $b\in\widetilde\AA$ is even,  and the von Neumann algebras $W^*(u)$ and $W^*(b)$ are $\ast$-free.
%\end{defi}
%Note that since $u$ is bounded, the product $u b$ is well-defined as an affiliated operator.
%
%An $R$-diagonal element can also be realized as $x=uh$ where $h\ge 0$ is affiliated
%with $\AA$ and $\ast$-free from $u$. In this case $h$ has the same distribution as
%$|x|=(x^\ast x)^{1/2}$, and $b$ has distribution $\mu_b=\mu_h^\sym$, the
%symmetrization of $\mu_h$ on $\mathbb R$, defined by
%\[
%\mu_h^\sym(A)=\frac12\,\mu_h(A\cap[0,\infty))+\frac12\,\mu_h((-A)\cap[0,\infty)),
%\qquad A\subset\mathbb R\ \text{Borel}.
%\]

\begin{defi}
An element $x\in\widetilde\AA$ is called \emph{$R$-diagonal} if it can be written as
$x=v b$, where $v\in\AA$ is a Haar unitary, $b\in\widetilde\AA$ is even, and the von Neumann algebras $W^*(v)$ and $W^*(b)$ are $\ast$-free.
\end{defi}

Note that since $v$ is bounded, the product $vb$ is well-defined as an affiliated operator.

Equivalently, an $R$-diagonal element can be written as $x=uh$ where $u\in\AA$ is a Haar unitary, $h=|x|:=(x^\ast x)^{1/2}\ge 0$, and $u$ is $\ast$-free from $h$. The even element $b$ in the first representation satisfies $\mu_b=\mu_h^{\mathrm{sym}}$, the symmetrization of $\mu_h$ on $\mathbb R$, defined by
\[
\mu_h^{\mathrm{sym}}(A)=\frac12\,\mu_h(A\cap[0,\infty))+\frac12\,\mu_h((-A)\cap[0,\infty)),
\qquad A\subset\mathbb R\ \text{Borel}.
\]

The following is a list of properties of $R$-diagonal variables, which we will use in proofs.

%----------------Tools for calculations-------------------------------

\begin{propo}[See  Proposition~3.6 in \cite{haagerup_larsen00} (bounded case) and Proposition~3.8 in \cite{haagerup_schultz2007} (affiliated case)]
\label{propo_product_R}
If both $x$ and $y$ are $R$-diagonal and $\ast$-free, then $xy$ is $R$-diagonal and
\[
\mu_{|xy|^2}=\mu_{|x|^2}\boxtimes \mu_{|y|^2}.
\]
\end{propo}

\begin{propo}[See  Proposition~3.10 \cite{haagerup_larsen00} (bounded case), Proposition~3.9 \cite{haagerup_schultz2007} (powers), and Proposition~3.6 \cite{haagerup_schultz2007} (inversion)]
\label{propo_powers_R}
\,
\begin{enumerate}
\item Let $x$ be $R$-diagonal. Then $x^k$ is $R$-diagonal for every integer $k\ge 1$, and
\[
\mu_{|x^k|^2}=\mu_{|x|^2}^{\boxtimes k}
:=\underbrace{\mu_{|x|^2}\boxtimes\cdots\boxtimes \mu_{|x|^2}}_{k\ \text{factors}}.
\]
\item If $x$ is $R$-diagonal and $\ker(x)=\{0\}$, then $x^{-1}$ is $R$-diagonal and
\[
\mu_{|x^{-1}|} = (\mu_{|x|})_{\mathrm{inv}},
\]
where $(\mu_{|x|})_{\mathrm{inv}}$ is the pushforward of $\mu_{|x|}$ under the inversion map
$t\mapsto t^{-1}$ on $(0,\infty)$.
\end{enumerate}
\end{propo}

\begin{propo}[See Proposition~3.5 \cite{haagerup_larsen00} (bounded case) and Proposition~3.11 \cite{haagerup_schultz2007} (affiliated case)]
\label{propo_additivity_symmetrizations}
If both $x$ and $y$ are $R$-diagonal and $\ast$-free, then $x+y$ is $R$-diagonal and
\[
\mu_{|x+y|}^{\sym}=\mu_{|x|}^{\sym}\boxplus \mu_{|y|}^{\sym}.
\]
\end{propo}

%-----------------------------Brown measure-------------------------

\subsubsection*{Brown measure}
The \emph{Brown measure} is an infinite-dimensional analogue of the eigenvalue distribution which applies both to normal and non-normal operators. 

 For $a \in \AA$ and $\lambda \in \bC$, let  
 \bal{
 L_{a}( \lambda) := \varphi(\log | a -\lambda 1|) 
, \text{ where } | a -\lambda 1| =\Big[( a-\lambda 1) ^{\ast
}( a-\lambda 1) \Big] ^{1/2},
}
and $1$ is the unit in $\AA$. Then, the \emph{Brown measure of $a$} is 
\bal{
\mu_a = \frac{1}{2\pi} \nabla^2  L_{a}( \lambda) \, dm_2(\lambda), 
}
where $dm_2(\lambda) = d \,\Re\lambda \, d \, \Im \lambda$, operator $\nabla^2$ is the Laplace operator, and the equality is meant in the sense of distributions.

Following \cite{haagerup_schultz2007}, we define 
 \[
 \mathcal{A}^{\Delta} = \{a \in \widetilde{\mathcal{A}} : \varphi(\log^+|a|) < \infty\}.
 \]
 Haagerup and Schultz showed that $\mathcal{A}^\Delta$ is a $\ast$-subalgebra of $\widetilde{\mathcal{A}}$ (Lemma 2.4, where $\ast$-closure is shown in the proof, and Propositions 2.5 and 2.6 in \cite{haagerup_schultz2007}) and that the Brown measure is well-defined for operators $a \in \mathcal{A}^\Delta$. A sufficient condition for $a \in \mathcal{A}^{\Delta}$ is that $a\in L^p(\AA,\varphi)$ for some $p > 0$, equivalently, 
 \begin{equation}
 \label{equ_sufficient_condition}
 \int_0^\infty t^p \, d\mu_{|a|}(t) < \infty.
 \end{equation}
See Remark 2.2 in \cite{haagerup_schultz2007}.

For $x \in \mathcal{A}^\Delta$, we denote its Brown measure by $\mu_x$; this notation is consistent with our earlier use of $\mu_a$ for the spectral measure of a self-adjoint $a$.

The fundamental tool for calculating the Brown measures of $R$-diagonal random variables is the following result. 
\begin{propo}[cf. Theorem 4.4 in \cite{haagerup_larsen00} and Theorem 4.17  in \cite{haagerup_schultz2007}]
\label{propo_HLS}
Let $x  = uh\in \AA^\Delta$ be an $R$-diagonal random variable, with a Haar unitary $u$ and $h = |x| \geq 0$, and let $\mu_x$ be its Brown measure.   Assume that $\mu_h \ne \delta_0$. 
\begin{enumerate}[label = (\Roman*)]
\item
The measure $\mu_x$ is rotationally invariant. If $\mu_x(\{0\}) = 0$ then 
\[
\mathrm{supp}(\mu_x) = \{\lambda \in \bC: \| x^{-1}\|_2^{-1} \leq \lambda \leq \|x\|_2)\},
\]
where $\|x\|_2 := \sqrt{\varphi(x^\ast x)}$. Otherwise, 
$\mathrm{supp}(\mu_x) = \{\lambda \in \bC:  \lambda \leq \|x\|_2)\}.$
\item For an atom at zero, 
\begin{equation}
\mu_x(\{0\}) = \mu_h(\{0\}).
\end{equation}

%\item
% $S$-transform $S_{h^2}$ is strictly decreasing  on $\big(\mu_h(\{0\}) - 1, 0\big)$ and 
%\bal{
%S_{h^2}\big(\big(\mu_h(\{0\}) - 1, 0\big)\big)
%= \begin{cases}
%\big(\|x\|_2^{-2}, \|x^{-1}\|_2^2\big), & \text{ if } \mu_h(\{0\}) = 0, \\
%\big(\|x\|_2^{-2}, \infty \big), & \text{ if } \mu_h(\{0\}) > 0.
%\end{cases}
%}

\item
 $S_{h^2}$ is non-increasing on $\big(\mu_h(\{0\}) - 1, 0\big)$, and strictly decreasing unless $\mu_h$ is a point mass. Moreover,
\bal{
S_{h^2}\big(\big(\mu_h(\{0\}) - 1, 0\big)\big)
= \begin{cases}
\big(\|x\|_2^{-2}, \|x^{-1}\|_2^2\big), & \text{ if } \mu_h(\{0\}) = 0 \text{ and } \mu_h \neq \delta_c, \\
\big\{\|x\|_2^{-2}\big\}, & \text{ if } \mu_h = \delta_c \text{ for some } c > 0, \\
\big(\|x\|_2^{-2}, \infty \big), & \text{ if } \mu_h(\{0\}) > 0.
\end{cases}
}

\item The Brown measure $\mu_x$ satisfies the following equation:
\begin{equation}
\label{equ_HL0}
\mu_x \Big(B\big(0, \big[S_{h^2}(t - 1)\big]^{-1/2}\big)\Big) = t 
\text { for } t \in (\mu_h(\{0\}), 1). 
\end{equation}
\end{enumerate}
\end{propo}

%%%%%%%%%%%%%%%%%%%%%%%%%%%%%%%
%
%%%%%%%%%%%%%%%%%%%%%%%%%%%%%%

\section{Proofs of Main Results}
\label{section_results}

\subsection*{Convolution semigroup}

Let $\MM_{BR}$ be the class of rotationally-invariant measures on $\bC$, which are Brown measures of $R$-diagonal elements in $\AA^\Delta$.  If $\mu_1, \mu_2 \in \MM_{BR}$ are the Brown measures of $\ast$-free $R$-diagonal elements $x_1$ and $x_2$, respectively, we define the measure $\mu_1 \boxplus_B \mu_2$ as the Brown measure of the $R$-diagonal element $x_1 + x_2$. 
This measure does not depend on the choice of $x_1$ and $x_2$ and $\boxplus_B$ is an associative and commutative operation on measures in $\MM_{BR}$. 

%-----------------------------Proof of lemma about the existence of the semigroup-----------

\begin{proof}[Proof of Proposition \ref{lemma_semigroup}]
First, we need to check that the Brown measure of $x_s$ exists. 
Let $x=uh\in\AA^\Delta$ and $\nu=\mu_h^{\mathrm{sym}}$. Then
$\int \log_+|t|\,d\nu(t)=\varphi(\log_+ h)<\infty$.

For $s\ge1$, realize $\nu^{\boxplus s}$ as the law of a free compression
$b_s=[P_s(sb)P_s]\in\widetilde{\AA}_{P_s}$, where $b$ is even self-adjoint
with law $\nu$ and $P_s$ is a projection free from $b$ with $\varphi(P_s)=1/s$.
Since $\AA^\Delta$ is a $*$-algebra and is stable under multiplication by
bounded elements (in particular by $P_s$), we have $b_s\in(\AA_{P_s})^\Delta$,
hence $x_s:=ub_s\in(\AA_{P_s})^\Delta$. Equivalently, $\int \log_+|t|\,d\nu(t)<\infty$ implies
$\int \log_+|t|\,d(\nu^{\boxplus s})(t)<\infty$ for all $s\ge1$. Therefore the Brown measure of $x_s$,
which we denote by $\mu^{\boxplus s}$, is well-defined.

\textbf{Property (\ref{semigroup_at_1}).} Clearly, (\ref{semigroup_at_1}) is satisfied because $x_1 = x$ and therefore $\mu^{\boxplus 1} = \mu$. 

\textbf{Property (\ref{semigroup_additivity}).}  Let $y_{s}$ and $y_{t}$ be two free $R$-diagonal variables that have the Brown measures $\mu^{\boxplus s}$ and $\mu^{\boxplus t}$. Then 
\bal{
 \mu_{|y_s|}^\sym = \mu_{|b_s|}^\sym = \mu_{b_s} = \nu^{\boxplus s}
}
Similarly, $\mu_{|y_t|}^\sym = \nu^{\boxplus t}$. Then, 
by Proposition \ref{propo_additivity_symmetrizations}, 
\bal{
 \mu_{|y_s + y_t|}^\sym = \mu_{|y_s|}^\sym \boxplus  \mu_{|y_t|}^\sym = \nu^{\boxplus s} \boxplus \nu^{\boxplus t} = \nu^{\boxplus (s + t)}.
}
Since the symmetrized measure of $|y_s + y_t|$ determines the Brown measure of $y_s + y_t$, we conclude that $y_s + y_t$ has the same Brown distribution as $u b_{s + t}$, where $b_{s + t}$ is an even variable with distribution~$\nu^{\boxplus (s + t)}$. It follows that 
$\mu^{\boxplus s} \boxplus_B \mu^{\boxplus t} = \mu^{\boxplus (s + t)}.$

\textbf{Property \ref{semigroup_continuity}.} 
(We use Theorem~\ref{theo_semigroup_unbounded} and Corollary~\ref{coro_no_atom}, which are proved below; their proofs depend only on Properties~(\ref{semigroup_at_1})--(\ref{semigroup_additivity}) and not on~(\ref{semigroup_continuity}).)
Let $x = uh$ with $h = |x| \geq 0$, and let $h_s \geq 0$ be such that $x_s = uh_s$ has Brown measure $\mu^{\boxplus s}$. Set $\delta := \mu_h(\{0\})$. By Proposition~\ref{propo_HLS}, $S_{h^2}(z)$ is analytic on $(\delta - 1, 0)$. By Theorem~\ref{theo_semigroup_unbounded},
\[
S_{h_s^2}(z) = \frac{1}{s}\frac{1 + z/s}{1 + z} S_{h^2}\Big(\frac{z}{s}\Big),
\]
which is well-defined for $z/s \in (\delta - 1, 0)$, i.e., for $z \in (s(\delta - 1), 0)$.

By Proposition~\ref{propo_HLS}, the radial cumulative distribution function of $\mu^{\boxplus s}$ satisfies
\[
F_s(r) := \mu^{\boxplus s}(B(0,r)) = t \quad \text{where} \quad r = \big[S_{h_s^2}(t-1)\big]^{-1/2},
\]
for $t \in (\delta_s, 1)$, where $\delta_s := \mu^{\boxplus s}(\{0\})$. By Corollary~\ref{coro_no_atom}, $\delta_s = \max(1 - s(1-\delta), 0)$, so $t - 1 \in (s(\delta - 1), 0)$ for $t \in (\delta_s, 1)$, which is precisely the domain where $S_{h_s^2}(t-1)$ is defined.

Since $S_{h^2}$ is continuous on $(\delta - 1, 0)$ and the prefactor is rational in $s$, the map $s \mapsto S_{h_s^2}(t-1)$ is continuous for each fixed $t$. Hence, for any $s_0 \in [1, \infty)$, the quantile functions $F_s^{(-1)}(t)$ converge pointwise to $F_{s_0}^{(-1)}(t)$ as $s \to s_0$. By standard results in probability theory, pointwise convergence of quantile functions implies weak convergence of the corresponding distributions. Since $\mu^{\boxplus s}$ is rotationally invariant, weak convergence of the radial distributions implies weak convergence of $\mu^{\boxplus s}$ on $\bC$.
\end{proof} 
 %-------------------Proof of the first theorem-----------------------

\subsection*{Proof of Theorem \ref{theo_semigroup_unbounded}}
 
\textbf{Domains for $R$- and $S$-transforms:}
In the remainder of Section~3, every identity involving $R_\mu$ or $S_\mu$
is understood to hold for $z$ in a domain where all transforms (and inverses) appearing
in the identity are well-defined. Concretely, for $R_\mu$ we work on a (possibly small)
truncated cone $D\subset\C_-$ containing $-i(0,\varepsilon)$, while for $S$-transforms
we work in an open neighborhood of $(-\varepsilon,0)$ (both for measures supported on
$[0,\infty)$ and for symmetric measures, in the sense of the symmetric $S$-transform).

 \begin{lemma}[$S$-transform of free additive convolution powers]
\label{lemma_S_semigroup}
Let $\nu \ne \delta_0$ be a measure on $\R$, which is either symmetric or supported on $\R^+ = [0, \infty)$, and let $\nu_s := \nu^{\boxplus s}$, $s \geq 1$. Then, 
$S_{\nu_s}(z) = \frac{1}{s} S_\nu\Big(\frac{z}{s}\Big),$
for all $z \in (-\eps, 0)$ with $\eps > 0$  small enough that both sides are defined.
\end{lemma}
\begin{proof}
By Theorem 2.5 in \cite{belinschi_bercovici04}, $R_{\nu_s}(z) = s R_{\nu}(z)$.  Note that 
\begin{align*}
z R_{\nu_s}(z) &= u \Leftrightarrow z R_{\nu}(z) = u /s, \text{ hence } \\
(z R_{\nu_s}\big)^{(-1)} (u) &=  (z R_{\nu}\big)^{(-1)} (u/s) 
\end{align*}
wherever the inverses are defined, in particular for $u \in (-\eps, 0)$, with sufficiently small $\eps > 0$. 
 We use formula (\ref{equ_S_HL}) and obtain:
\bal{
S_{\nu_s}(u) &= \frac{(z R_{\nu_s}\big)^{(-1)} (u)}{u} 
= \frac{(z R_{\nu}\big)^{(-1)} (u/s) }{u} = \frac{(u/s) S_\nu (u/s) }{u}
= \frac{1}{s} S_{\nu}(u/s). 
}
\end{proof}

\begin{lemma}[Arizmendi--P\'erez-Abreu squaring identity]
\label{lem:APA_squaring}
Let $\mu$ be a symmetric probability measure on $\R$, and let $S_\mu$ denote the $S$-transform of $\mu$
\emph{in the sense of} \cite{arizmendi_perez-abreu2009}.
Then there exist $\varepsilon>0$ and an open set $U\subset i\C^{+}$ with $(-\varepsilon,0)\subset U$ such that
$S_\mu$ and $S_{\mu_{sq}}$ are well-defined on $U$, and for every $z\in U$,
\begin{equation}
\label{equ_S_mu_sq}
S_{\mu_{sq}}(z) \;=\; \frac{z}{1+z}\,\big[S_\mu(z)\big]^2,
\end{equation}
where $\mu_{sq}$ is the push-forward of $\mu$ under the map $t\mapsto t^2$.
\end{lemma}

\noindent\emph{Note.} In \eqref{equ_S_mu_sq} the square is taken on the \emph{value} $S_\mu(z)$, not on the \emph{argument} $z$.
For a proof, see Theorem~6 in \cite{arizmendi_perez-abreu2009}.

%\begin{lemma}[Arizmendi--P\'erez-Abreu squaring identity]
%For a symmetric measure $\mu$ on $\R$, one has 
%\begin{equation}
%\label{equ_S_mu_sq}
%S_{\mu_{sq}}(z) = \frac{z}{1 + z} \Big[S_\mu(z)\Big]^2,
%\end{equation}
%where $\mu_{sq}$ is the push-forward of the measure $\mu$ under the transformation $t \to t^2$.
%\end{lemma}
%For the proof see Theorem 6 in \cite{arizmendi_perez-abreu2009}.

 \begin{proof}[Proof of Theorem \ref{theo_semigroup_unbounded}]
 We can write $x_s = u b_s$ where $b_s$ is even and $b_s^2 = h_s^2$. If the measure of $b = b_1$ is $\nu$, then the measure of $b_s$ is $\nu_s = \nu^{\boxplus s}$ by the definition of the semigroup of Brown measures. Therefore, by Lemma \ref{lemma_S_semigroup}, 
 $
 S_{b_s}(z) = \frac{1}{s} S_b(z/s).
 $
By (\ref{equ_S_mu_sq}),  it follows that 
 \bal{
 S_{b_s^2}(z) = \frac{z}{1 + z} \big[S_{b_s}(z)\big]^2 = 
 \frac{z}{1 + z} \frac{1}{s^2} \big[S_b(z/s)\big]^2.
 }
 On the other hand, (\ref{equ_S_mu_sq}) also implies that 
 \bal{
 S_{b^2}(z/s) = \frac{z/s}{1 + z/s} \big[S_b(z/s)\big]^2. 
 }
 Therefore, 
 \bal{
  S_{b_s^2}(z) &=  \frac{z}{1 + z} \frac{1}{s^2} \frac{1 + z/s}{z/s} S_{b^2}(z/s) = \frac{1}{s}\frac{1 + z/s}{1 + z} S_{b^2}(z/s).
 }
 \end{proof}

%%%%%%%%%%%%%%%%%%%%%%%%%%%%%%
%%%%%%%%%%%
%%%%%%%%%%%%%%%%%%%%%%%%%%%%%%
%\subsection{Proof of Theorem \ref{theo_semigroups}}
\subsection*{Proof of Theorem \ref{theo_semigroups}}
 We start with several auxiliary results.

\begin{lemma}[$S$-transform under free compression]
\label{lemma_compression_S}
Let $a\neq 0$ be a self-adjoint element in $\widetilde\AA$, and assume that either $a\ge 0$ or $a$ is even.
Let $\nu:=\mu_a$ denote the $\varphi$-distribution of $a$.
Let $P_s$ be a projection $\ast$-free from $a$ with $\varphi(P_s)=\frac1s\in(0,1]$ and let
$\tilde\varphi([P_s a P_s]) := s\,\varphi(P_s a P_s)$ be the compressed trace on $\AA_{P_s}$.
Let $\nu_s$ denote the $\tilde\varphi$-distribution of the compressed/scaled variable
$[P_s (s a) P_s]\in\widetilde\AA_{P_s}$.

Then, for all $z$ for which both sides are defined,
\[
S_{\nu_s}(z)=\frac{1}{s}\,S_{\nu}\!\left(\frac{z}{s}\right).
\]
Here $S_\nu$ denotes the usual $S$-transform in the case $a\ge 0$, and the symmetric $S$-transform
in the case that $a$ is even.
\end{lemma}

\begin{proof}
By  \cite[Corollary 1.14]{nica_speicher96} the compression of the self-adjoint random variable $sa$, $[P_s(sa)P_s]$ has the measure $\nu_s = \nu^{\boxplus s}$. Then, apply Lemma \ref{lemma_S_semigroup}.
\end{proof}

\begin{lemma}[$R$-diagonal conjugation frees a projection]
\label{lemma_xPx_P}
Let $x = u h \in \widetilde \AA$ be an $R$-diagonal random variable and $P$ be a projection, $\ast$-free of $\{u, h\}$. Then $x P x^\ast$ is free from $P$. 
\end{lemma}

\begin{proof}
Let $\mathcal B:=W^\ast(P,h)$ (equivalently, the von Neumann algebra generated by $P$ and the spectral projections of $h$). 
Since $u$ is a Haar unitary $\ast$-free from $\mathcal B$, Lemma~3.7 of \cite{haagerup_larsen00} implies that $\mathcal B$ and $u\mathcal B u^\ast$ are $\ast$-free.
Moreover, $P\in\mathcal B$ and $hPh$ is affiliated with $\mathcal B$, hence for every $\lambda\in\C^+$,
\[
(\lambda-xPx^\ast)^{-1}=u(\lambda-hPh)^{-1}u^\ast\in u\mathcal B u^\ast.
\]
Therefore the resolvent algebra of $xPx^\ast$ lies in $u\mathcal B u^\ast$ and is $\ast$-free from $P\in\mathcal B$, which means $xPx^\ast$ is free from $P$.
\end{proof}

\textbf{Convention.} Whenever $a \in \widetilde \AA_{P_s}$, its distribution and all analytic transforms (in particular, $S_a$) are computed with respect to the normalized trace $\tilde \varphi \big([P_s a P_s]\big) = s \varphi(P_s a P_s)$. 

\begin{lemma}[Compressed $S$-transform of $xP_sx^\ast$ (R-diagonal case)]
\label{lemma_tildeS}
Let $x = u h \in \widetilde \AA$ be an $R$-diagonal random variable with $u$ Haar unitary and $h \geq 0$. Let $P_s$, $s \geq 1$ be a projection, $\ast$-free of $\{u, h\}$, with $\varphi(P_s) = 1/s$. Let $y_s =  [P_s s(x P_s x^\ast) P_s] \in \widetilde \AA_{P_s}$.  Then, 
\bal{
S_{y_s}(z) = \frac{1 + z/s}{1 + z} S_{h^2}(z/s). 
}
\end{lemma}

\begin{proof}
By Lemma~\ref{lemma_xPx_P}, the positive operator $a:=xP_sx^\ast$ is free from $P_s$.
Hence Lemma~\ref{lemma_compression_S} applies to $a$ and yields 
\begin{equation}
\label{tildeS_1}
 S_{y_s}(z) = \frac{1}{s} S_{x P_s x^\ast}(z/s).
\end{equation}

Since $x=uh$ with $h\ge 0$, we have $xP_sx^\ast=u(hP_sh)u^\ast$.
Because $\varphi$ is a trace, for any bounded Borel function $f$ one has
$f(uAu^\ast)=u f(A)u^\ast$ and hence $\varphi(f(uAu^\ast))=\varphi(f(A))$.
In particular, $xP_sx^\ast$ and $hP_sh$ have the same $\varphi$-distribution.
Moreover,
$
hP_sh=(h^2)^{1/2}P_s(h^2)^{1/2},
$ 
and $h^2$ is free from $P_s$. Therefore, by multiplicativity of the $S$-transform for free
positive variables,
\[
S_{xP_sx^\ast}(z)=S_{hP_sh}(z)=S_{h^2}(z)\,S_{P_s}(z)
=\frac{1+z}{1/s+z}\,S_{h^2}(z),
\]
where we used the fact that for projection $P_s$, the $S$-transform is  $(1 + z)/(1/s + z)$. After we substitute this expression in (\ref{tildeS_1}), we get:
\bal{
S_{y_s}(z) = \frac{1}{s} \frac{1 + z/s}{1/s + z/s} S_{h^2}(z/s)
=  \frac{1 + z/s}{1 + z} S_{h^2}(z/s)
}
\end{proof}

\begin{proof}[Proof of Theorem \ref{theo_semigroups}]
Following Definition \ref{defi_semigroup} for $\mu^{\boxplus s}$, let $x_s = u b_s$ be an $R$-diagonal variable with even $b_s$ that has measure  $\nu^{\boxplus s}$. 
By  (\ref{equ_HL0}), the distribution of $\mu^{\boxplus s}$ is determined by the $S_{|x_s|^2}(z) =   S_{b_s^2}(z)$, and
by Theorem \ref{theo_semigroup_unbounded}, 
\bal{
S_{b_s^2}(z)  = \frac{1}{s} \frac{1 + z/s}{1 + z} S_{|x|^2}\Big(\frac{z}{s}\Big)
}

By Lemma \ref{lemma_tildeS} and scaling relation $S_{ta}(z) = \frac{1}{t}S_a(z)$, for $\pi_s(x) = [P_s x P_s] \in (\widetilde \AA_{P_s}, \tilde \varphi)$ we have
\[
\tilde S_{|s\pi_s(x)|^2}(z) = \tilde S_{s y_s}(z) =  \frac{1}{s} \frac{1 + z/s}{1 + z} S_{|x|^2}\Big(\frac{z}{s}\Big),
\]
 and we conclude that $\mu^{\boxplus s}$ equals the Brown measure of  $s\pi_s(x)$.

\end{proof}

%%%%%%%%%%%%%%%%%%%%%%%%
%
%
%%%%%%%%%%%%%%%%%%%%%%%%%
%%%%%%%%%%%%%%%%%%%%%%%%%%%%%%
%%%%%%%%%%%
%%%%%%%%%%%%%%%%%%%%%%%%%%%%%%

\subsection*{Proof of Theorem \ref{theo_convergence_compressions}}

We prove Theorem  \ref{theo_convergence_compressions} as a consequence of Theorems \ref{theo_semigroup_unbounded}  and \ref{theo_semigroups}. 

\begin{proof}[Proof of Theorem \ref{theo_convergence_compressions}] 
Let $\mu$ be the Brown measure of $x$ and define 
\[
\widehat\mu_s(A) := \mu^{\boxplus s}(\sqrt{s}\,A).
\]
By Theorem \ref{theo_semigroups}, the Brown measure of $\sqrt{s}\,\pi_s(x)$ equals $\widehat\mu_s$, so it suffices to show that $\widehat\mu_s$ converges weakly to the uniform probability measure on the unit disk.

Let $x_s = u h_s$ be an $R$-diagonal variable with Brown measure $\mu^{\boxplus s}$. Then $s^{-1/2}u h_s$ has Brown measure $\widehat\mu_s$. 
By Theorem \ref{theo_semigroup_unbounded} and the scaling property of the $S$-transform,
\[
 S_{s^{-1}h_s^2}(z) = \frac{1 + z/s}{1 + z}\, S_{h^2}\!\Big(\frac{z}{s}\Big).
\]
By Proposition~\ref{propo_HLS}(III), $S_{h^2}(z)$ is continuous and strictly decreasing on $(-\varepsilon,0)$, and 
$\lim_{z\uparrow 0} S_{h^2}(z) = \|x\|_2^{-2}=1$. 
Fix $z\in(-1,0)$. For $s$ large enough, $z/s\in(-\varepsilon,0)$, hence $S_{h^2}(z/s)\to 1$ and therefore
\[
S_{s^{-1}h_s^2}(z)\to \frac{1}{1+z}.
\]
Consequently, for $t\in(0,1)$ we have $S_{s^{-1}h_s^2}(t-1)\to 1/t$, and by Proposition~\ref{propo_HLS}(I),(IV) this implies that $\widehat\mu_s$ converges to the uniform probability measure on the unit disk.
\end{proof}

Here are some other remarkable consequences of Theorem \ref{theo_semigroup_unbounded}.

\begin{coro}
\label{coro_support}
Suppose the Brown measure $\mu$ has no atom at $0$ (i.e., $\mu(\{0\}) = 0$) and suppose that $\mu$ is supported on a ring $R(a, b)$, where $0 \leq a \leq b \leq \infty$. Then, for each $s > 1$, $\mu^{\boxplus s}$ is supported on the disk $B(0, \sqrt{s} b)$.
\end{coro}
The main point of this corollary is that even if $\mu$ is supported on a ring with the inner radius $a > 0$, the measure $\mu^{\boxplus(1 + \eps)}$ is supported on a disk,  for every $\eps > 0$. (Observe that if $\mu$ has an atom at $0$, then $\mu$ itself is supported on a disk by Proposition \ref{propo_HLS}.)

\begin{proof}
By applying Proposition \ref{propo_HLS}, we have 
\begin{equation}
\label{equ_HL}
\mu^{\boxplus s} \Big(B\Big(0, \frac{1}{\sqrt{S_{|x_s|^2}(t - 1)}}\Big)\Big) = t 
\text { for } t \in (0, 1]. 
\end{equation}
By our assumption on the Brown measure $\mu$, the $S$-transform $S_{|x_1|^2}(z)$ is analytic on $(-1, 0)$.  Then formula (\ref{semigroup_S}) in Theorem \ref{theo_semigroup_unbounded} implies that $S_{|x_s|^2(z)}$ is also analytic on $(-1, 0)$, and that 
\bal{
\lim_{z \downarrow -1} S_{|x_s|^2}(z) = +\infty \text{ and } 
\lim_{z \uparrow 0} S_{|x_s|^2}(z) = \frac{1}{s} \lim_{z \uparrow 0} S_{|x|^2}(z).
}
Then, formula (\ref{equ_HL}) implies the claim of the corollary. 
\end{proof} 

The second corollary shows that for a sufficiently large $s$ the measure $\mu^{\boxplus s}$ has no atom at zero. 
\begin{coro}
\label{coro_no_atom}
If $\mu(\{0\}) = \delta > 0$, then 
\bal{
\mu^{\boxplus s} (\{0\}) =
\begin{cases}
1 - s ( 1- \delta),& \text{ if }  1 \leq s < \frac{1}{1 - \delta},  \\
0, & \text{ otherwise.}
\end{cases}
}
\end{coro}

\begin{proof}
Let $\mu^{\boxplus s} (\{0\}) = \delta_s$. Then, by Proposition \ref{propo_HLS} and by Proposition 6.8 in \cite{bercovici_voiculescu93} , $S_{|x_s|^2}(z)$ is analytic on $(\delta_s - 1, 0)$ and $\lim_{z \downarrow (\delta_s - 1)}  S_{|x_s|^2}(z) = + \infty$. Then, formula (\ref{semigroup_S}) in Theorem \ref{theo_semigroup_unbounded} implies that $(\delta_s - 1)/s = \delta - 1$ for $s < (1 - \delta)^{-1}$, from which we calculate $\delta_s$ and obtain the conclusion of the corollary for the case $s < (1 - \delta)^{-1}$. If $s \geq (1 - \delta)^{-1}$, then $S_{|x_s|^2}(z)$ is analytic on $(-1, 0)$, hence $\mu^{\boxplus s}(\{0 \}) = 0$. 
\end{proof}

This result can also be obtained by a straightforward combination of Proposition  \ref{propo_HLS}(II) and \cite[Theorem 3.1]{belinschi_bercovici04}.

\subsection*{A proof of Theorem \ref{theo_stable_main}}
\label{section_stable_measures}

Recall that a Brown measure $\mu \ne \delta_0 \in \MM_{BR}$  is \emph{stable} if 
\bal{
(\mu^{\boxplus k}) (A) = \mu(\alpha_k A),
}  
with $\alpha_k > 0$, for all $k \in \N$   and all Borel sets $A \subset \bC$. 
The stability property implies that $\mu$ is an infinitely divisible measure with respect to $\boxplus_B$, and that
$\mu^{\boxplus s} (A) = \mu(\alpha_s A)$,
with some $\alpha_s > 0$, for  all $s > 0$ and all Borel $A \subset \bC$. 

Indeed, for $a>0$ and a probability measure $\mu$ on $\C$, denote by $D_a\mu$ the dilation
\[
(D_a\mu)(A):=\mu(aA),\qquad A\subset\C \ \text{Borel}.
\]
Note that $D_a$ is compatible with $\boxplus_B$: if $\mu_1,\mu_2\in\MM_{BR}$ then
\begin{equation}\label{eq:dilation_covariance}
D_a(\mu_1\boxplus_B\mu_2)= (D_a\mu_1)\boxplus_B(D_a\mu_2).
\end{equation}
This is because if $x_1,x_2$ are $\ast$-free $R$-diagonal variables with Brown measures
$\mu_1,\mu_2$, then $a(x_1+x_2)$ has Brown measure $D_a(\mu_1\boxplus_B\mu_2)$ and
also equals $(ax_1)+(ax_2)$, whose Brown measure is $(D_a\mu_1)\boxplus_B(D_a\mu_2)$.

\begin{lemma}[Integer stability $\Rightarrow$ $\boxplus_B$-infinite divisibility]\label{lem:stable_inf_div}
Let $\mu\in\MM_{BR}$, $\mu\neq\delta_0$. Assume that for every integer $n\ge 1$ there
exists $\alpha_n>0$ such that
\begin{equation}\label{eq:stable_integer}
\mu^{\boxplus n}=D_{\alpha_n}\mu .
\end{equation}
Then:
\begin{enumerate}
\item $\alpha_{mn}=\alpha_m\alpha_n$ for all integers $m,n\ge 1$.
\item $\mu$ is $\boxplus_B$-infinitely divisible. In particular, for every $n\ge 1$ one can define
\[
\mu^{\boxplus 1/n}:=D_{\alpha_n^{-1}}\mu,
\]
and this satisfies $(\mu^{\boxplus 1/n})^{\boxplus n}=\mu$.
\item For rationals $s=p/q>0$ (with $p,q\in\N$) one may define the fractional power by
\[
\mu^{\boxplus p/q}:=(\mu^{\boxplus 1/q})^{\boxplus p},
\]
and then the stability relation extends with
\begin{equation}\label{eq:stable_rational}
\mu^{\boxplus p/q}=D_{\alpha_p\alpha_q^{-1}}\mu .
\end{equation}
\end{enumerate}
\end{lemma}

\begin{proof}
(1) Using \eqref{eq:dilation_covariance} and \eqref{eq:stable_integer},
\[
\mu^{\boxplus mn}=(\mu^{\boxplus n})^{\boxplus m}=(D_{\alpha_n}\mu)^{\boxplus m}
= D_{\alpha_n}(\mu^{\boxplus m})=D_{\alpha_n}D_{\alpha_m}\mu=D_{\alpha_m\alpha_n}\mu,
\]
hence $\alpha_{mn}=\alpha_m\alpha_n$ (since $\mu\neq\delta_0$, dilation factors are unique).
(2) Let $\nu_n:=D_{\alpha_n^{-1}}\mu$. Then
\[
\nu_n^{\boxplus n}=D_{\alpha_n^{-1}}(\mu^{\boxplus n})
=D_{\alpha_n^{-1}}D_{\alpha_n}\mu=\mu.
\]
(3) By definition and \eqref{eq:dilation_covariance},
\[
\mu^{\boxplus p/q}=(D_{\alpha_q^{-1}}\mu)^{\boxplus p}
= D_{\alpha_q^{-1}}(\mu^{\boxplus p})
= D_{\alpha_q^{-1}}D_{\alpha_p}\mu.
\]
\end{proof}

\begin{remark}\label{rem:s_lt_1}
For a general $\mu\in\MM_{BR}$, the semigroup $\mu^{\boxplus s}$ is defined (in this paper) for $s\ge 1$.
The existence of fractional powers $s<1$ requires $\boxplus_B$-infinite divisibility.
Lemma~\ref{lem:stable_inf_div} shows that this infinite divisibility is automatic under the stability
assumption \eqref{eq:stable_integer}.
\end{remark}

We split the proof of Theorem \ref{theo_stable_main} into two parts. In Proposition \ref{theo_stable_1}, we show that if a Brown measure of an $R$-diagonal random variable $x$ is stable then $x$ must satisfy the property $S_{|x|^2}(z) = c \frac{(-z)^\beta}{1 + z}$, and in Proposition  \ref{theo_stable_2} we show that if element $x$ satisfies this property with $\beta \geq 0$, then its Brown measure exists and is stable. Theorem \ref{theo_stable_main} follows by the relation between Brown measures and $S$-transforms formulated in Proposition \ref{propo_HLS}.

\begin{remark}[Branch convention]
\label{rem:branch_minus_z_beta}
Throughout this section we use the principal branch of the logarithm
\[
\Log w = \ln|w| + i\Arg(w), \qquad \Arg(w)\in(-\pi,\pi),
\]
defined on $\mathbb C\setminus(-\infty,0]$.
Accordingly, for $\beta\in\mathbb R$ and $z\in\mathbb C\setminus[0,\infty)$ we define
\[
(-z)^\beta := \exp\!\big(\beta\,\Log(-z)\big).
\]
In particular, if $z\in(-1,0)$ then $(-z)^\beta\in(0,\infty)$ is the usual real power.
\end{remark}

%--------Proposition: S-transform of stable Brown measures ------------

\begin{propo}
\label{theo_stable_1}
If $x \in \AA^\Delta$ is an $R$-diagonal random variable with a stable Brown measure $\mu \ne \delta_0$, then
\begin{equation}
\label{equ_S_stable}
S_{|x|^2}(z) = c \frac{(-z)^\beta}{1 + z},
\end{equation}
for some $c > 0$ and $\beta \ge 0$.
\end{propo}

\begin{proof}
Let $S(z):=S_{|x|^2}(z)$.

\smallskip\noindent
\textbf{Step 0: Extension from integers to rationals.}
By assumption, for each integer $n\ge 1$ there exists $\alpha_n>0$ such that
$\mu^{\boxplus n}=D_{\alpha_n}\mu$. By Lemma~\ref{lem:stable_inf_div}, this integer
stability implies that $\mu$ is $\boxplus_B$-infinitely divisible, and the scaling
relation extends to all positive rationals: for $s=p/q\in\mathbb Q_+$ (with $p,q\in\mathbb N$),
we have
$\mu^{\boxplus s}=D_{\alpha_s}\mu,$ where $\alpha_s:=\alpha_p/\alpha_q.$

For rational $s\ge 1$, set $x_s:=\alpha_s^{-1}x$, so that
$|x_s|^2=\alpha_s^{-2}|x|^2$ and therefore, by the scaling rule for the $S$-transform,
\[
S_{|x_s|^2}(z)=\alpha_s^{2}\,S(z)=:\lambda_s\,S(z),
\qquad \lambda_s>0.
\]
By stability, the Brown measure of $x_s$ equals $\mu^{\boxplus s}$, and by
Theorem~\ref{theo_semigroup_unbounded} (applied to $x$ and parameter $s$) we obtain the functional equation
\begin{equation}
\label{eq:functionalS}
\lambda_s\,S(z)=\frac{1}{s}\,\frac{1+z/s}{1+z}\,S(z/s),
\qquad s\in\mathbb Q_+,\ s\ge 1.
\end{equation}

Define
$
g(z):=(1+z)\,S(z).
$
Then \eqref{eq:functionalS} is equivalent to
\begin{equation}
\label{eq:functionalg}
\lambda_s\,g(z)=\frac{1}{s}\,g(z/s),
\qquad s\in\mathbb Q_+,\ s\ge 1.
\end{equation}
Now make the change of variables $t:=1/s\in(0,1]\cap\mathbb Q$ and rewrite \eqref{eq:functionalg} as
\begin{equation}
\label{eq:g_scaling}
g(tz)=\gamma(t)\,g(z),
\qquad t\in(0,1]\cap\mathbb Q,
\end{equation}
where
$
\gamma(t):=\frac{\lambda_{1/t}}{t}>0.
$

\smallskip\noindent
\textbf{Step 1: $\gamma$ is multiplicative on $(0,1]\cap\mathbb Q$.}
Applying \eqref{eq:g_scaling} twice gives, for $t_1,t_2\in(0,1]\cap\mathbb Q$ with $t_1t_2\in(0,1]$,
$
g(t_1t_2z)=\gamma(t_1)\,g(t_2z)=\gamma(t_1)\gamma(t_2)\,g(z),
$
while also $g(t_1t_2z)=\gamma(t_1t_2)\,g(z)$.  Since $g$ is not identically zero, we obtain
\[
\gamma(t_1t_2)=\gamma(t_1)\gamma(t_2),
\qquad t_1,t_2\in(0,1]\cap\mathbb Q.
\]

\smallskip\noindent
\textbf{Step 2: Extension to $(0,1]$ and the form $\gamma(t)=t^\beta$.}
Fix any $z_0\in(-\varepsilon,0)$ in the domain where $S$ is defined.
Then $S(z_0)>0$ and $1+z_0>0$, hence $g(z_0)>0$. Define $\Gamma:(0,1]\to(0,\infty)$ by
\[
\Gamma(t):=\frac{g(tz_0)}{g(z_0)}.
\]
Since $g$ is analytic, $\Gamma$ is continuous on $(0,1]$. Moreover, by \eqref{eq:g_scaling},
$\Gamma(t)=\gamma(t)$ for all $t\in(0,1]\cap\mathbb Q$.

The multiplicativity of $\gamma$ on $(0,1]\cap\mathbb Q$ means that
$\Gamma(t_1t_2)=\Gamma(t_1)\Gamma(t_2)$ holds for all $t_1,t_2\in(0,1]\cap\mathbb Q$.
Since $(0,1]\cap\mathbb Q$ is dense in $(0,1]$ and $\Gamma$ is continuous,
this multiplicativity extends to all $t_1,t_2\in(0,1]$.

Now set $h(u):=\log\Gamma(e^{-u})$ for $u\ge 0$.  The multiplicativity of $\Gamma$ becomes additivity of $h$,
and continuity gives $h(u)=\beta u$ for some $\beta\in\mathbb R$.  Equivalently,
$\Gamma(t)=t^\beta$, for all $t\in(0,1]$.
Since $\Gamma(t)=g(tz_0)/g(z_0)$ for all $t\in(0,1]$, we obtain
\begin{equation}
\label{eq:g_homogeneous}
g(tz)=t^\beta g(z),\qquad t\in(0,1],\ z\in(-\varepsilon,0).
\end{equation}

\smallskip\noindent
\textbf{Step 3: Determine the analytic form of $g$.}
Fix once and for all the branch $(-z)^\beta:=\exp(\beta\,\Log(-z))$ on $\mathbb C\setminus[0,\infty)$,
where $\Log$ is the principal logarithm (so $(-z)^\beta>0$ for $z\in(-\infty,0)$).
Define
\[
\widetilde g(z):=\frac{g(z)}{(-z)^\beta}.
\]
Then \eqref{eq:g_homogeneous} implies $\widetilde g(tz)=\widetilde g(z)$ for all $t\in(0,1]$.
In particular, for $z\in(-\varepsilon,0)$ the function $\widetilde g$ is constant on the whole interval
(by varying $t$), hence by analyticity it is constant on the connected domain of definition.
Therefore $\widetilde g(z)\equiv c$ for some constant $c\in\mathbb C$, and so
\[
g(z)=c\,(-z)^\beta,
\qquad\text{i.e.}\qquad
S(z)=c\,\frac{(-z)^\beta}{1+z}.
\]

\smallskip\noindent
\textbf{Step 4: $c>0$ and $\beta\ge 0$.}
For real $z\in(-\varepsilon,0)$ we have $S(z)>0$, $1+z>0$, and $(-z)^\beta>0$ by the branch choice.
Thus
\[
c = S(z)\,\frac{1+z}{(-z)^\beta} > 0.
\]
Finally, by Proposition~\ref{propo_HLS}(III) (in particular, the existence of a finite limit
$\lim_{z\uparrow 0} S(z)\in[0,\infty)$ for a nonzero positive variable), we have
$\lim_{z\uparrow 0} g(z)\in[0,\infty)$ as well.  Since $g(z)=c(-z)^\beta$ with $c>0$,
this is possible only when $\beta\ge 0$.
\end{proof}

%------------------------------Every $(c, \beta)$ is realizable ----------------------------

The next proposition shows that, conversely, every pair $(c, \beta)$, with $c > 0$ and $\beta \geq 0$, corresponds to a stable Brown measure. Since $c$ is simply a scaling parameter, it is enough to show this statement holds for  $c = 1$ and $\beta \geq 0$. 
\begin{propo}
\label{theo_stable_2}
For every $\beta \geq 0$, there exists an $R$-diagonal $x \in \AA^\Delta$, so that 
\begin{equation}
\label{stable_S}
S_{|x|^2}(z) = \frac{(-z)^\beta}{1 + z}.
\end{equation}
The random variable $x$ has the Brown measure, denoted $\mu_\beta$, which is stable. 
\end{propo}
\begin{proof}
By Proposition 3 in \cite{haagerup_moller2013}, for every $\beta \geq 0$,  $(-z)^\beta/(1 + z)$ is a valid $S$-transform of a measure $\nu_\beta$ on  $(0, +\infty)$. Let $h \in \tilde \AA $, $h \geq 0$, and the measure of $h^2$ be $\nu_\beta$. Define $x = u h$, where $u$ is  a Haar unitary, $\ast$-free of $h$. Then, $S_{|x|^2}(z) = (-z)^\beta/(1 + z)$. We claim that $x \in \AA^\Delta$ and therefore the Brown measure  for $x$ is well defined. We denote it $\mu_\beta$. 

Indeed, for $\beta = 0$, $x$ is the standard circle random variable. For $\beta > 0$, by Theorem 3 in \cite{haagerup_moller2013}, the moments of $\nu_\beta$ are given by the formula 
\bal{
\int_0^\infty x^\gamma \, d\nu_\beta(x) 
= \begin{cases}
\frac{\sin(\pi \gamma)}{\pi \gamma}
\frac{\Gamma(1 + 2\gamma) \Gamma(1 - (1 + \beta) \gamma)}{\Gamma(2 + (1 - \beta)\gamma)},& -\frac{1}{2} < \gamma < \frac{1}{1 + \beta}, \\
\infty, & \text{ otherwise.} 
\end{cases}
}

Since $\int_0^\infty x^\gamma \, d\nu_\beta(x) < \infty$  for some $\gamma > 0$, it follows by (\ref{equ_sufficient_condition}) that $x \in \AA^\Delta$ and the Brown measure  for $x$ is well defined. 

We claim that $\mu_\beta$ is stable. Indeed, by formula (\ref{semigroup_S}), 
$
S_{|x_s|^2}(z) = s^{-1 - \beta} (-z)^\beta (1 + z)^{-1}. 
$
This implies that $s^{-(1 + \beta)/2} x_s$ has the same Brown measure as $x$. Then 
\bal{
\mu_\beta^{\boxplus s}(A) = \mu_\beta(s^{-\frac{1 + \beta}{2}}A),
} 
which implies the stability of $\mu_\beta$. 
\end{proof}
%

%-----------------------------Properties of stable Brown measures------------

\subsection*{Properties of stable Brown measures}

What can be said about properties of the measure of squared singular values $\nu_\beta$ and the Brown measure $\mu_\beta$?

By Proposition 6 in \cite{haagerup_moller2013}, the measure $\nu_\beta$ has a continuous density $f_{\nu,\beta}(x)$, $x > 0$, with respect to the Lebesgue measure on $\R$. In general, the density is not easy to describe explicitly, see a characterization in Theorem 6 of \cite{haagerup_moller2013}. 

For the Brown measure $\mu_\beta$, let $F_\beta(r) = \mu_\beta\big(B(0, r)\big)$ and let $F_\beta^{(-1)}(t)$ denote the functional inverse of $F_\beta(r)$. By formulae (\ref{equ_HL}) and (\ref{stable_S}), 
\begin{equation}
\label{equ_r_t}
r = F_\beta^{(-1)}(t) = \frac{t^{1/2}}{(1 - t)^{\beta/2}},
\end{equation}
and $F_\beta(r)$ can be obtained by inverting this function.

\begin{propo}[Power-law tail of $\mu_\beta$]
\label{propo_tails}
Let $\beta>0$ and let $\mu_\beta$ be the stable Brown measure from Definition \ref{defi_mu_beta}.
Then $\mu_\beta$ is absolutely continuous with respect to planar Lebesgue measure $dm_2$.
Moreover, $\mu_\beta$ is rotationally invariant, so there exists a measurable function
$\rho_\beta:(0,\infty)\to[0,\infty)$ such that
\[
d\mu_\beta(z)=\rho_\beta(|z|)\,dm_2(z).
\]
As $r\to\infty$,
\[
\rho_\beta(r)\sim \frac{1}{\pi\beta}\, r^{-(2+2/\beta)}.
\]
Equivalently, the density of the radius variable $R=|Z|$ (where $Z\sim\mu_\beta$),
\[
f_\beta(r):=\frac{d}{dr}\mu_\beta(B(0,r))=2\pi r\,\rho_\beta(r),
\]
satisfies
\[
f_\beta(r)\sim \frac{2}{\beta}\, r^{-(1+2/\beta)}.
\]
\end{propo}
\begin{proof}
Recall that $F_\beta(r)=\mu_\beta(B(0,r))$ and that
\[
F_\beta^{(-1)}(t)=\frac{t^{1/2}}{(1-t)^{\beta/2}},\qquad t\in(0,1).
\]
Let $t=F_\beta(r)$ and consider $r\to\infty$ (equivalently $t\uparrow 1$).
From the above formula we have $r\sim (1-t)^{-\beta/2}$, hence
\[
1-t\sim r^{-2/\beta}.
\]
Differentiate $r(t)=t^{1/2}(1-t)^{-\beta/2}$:
\[
\frac{dr}{dt}
=\frac{1}{2}t^{-1/2}(1-t)^{-\beta/2}
+\frac{\beta}{2}t^{1/2}(1-t)^{-\beta/2-1}
\sim \frac{\beta}{2}\, r\, (1-t)^{-1}
\sim \frac{\beta}{2}\, r^{1+2/\beta}.
\]
Therefore
\[
F_\beta'(r)=\frac{dt}{dr}\sim \frac{2}{\beta}\, r^{-(1+2/\beta)}.
\]
Since $\mu_\beta$ is rotationally invariant, $F_\beta'(r)=2\pi r\,\rho_\beta(r)$ a.e.,
so
\[
\rho_\beta(r)\sim \frac{1}{\pi\beta}\, r^{-(2+2/\beta)}.
\]
\end{proof}

\begin{propo}[Moments and integrability]\label{propo_moments}
Let $\mu_\beta$ be the stable Brown measure from Definition \ref{defi_mu_beta}, with parameter $\beta\ge 0$, and let $k\ge 0$. Then
\[
\int_{\C} |z|^k\, d\mu_\beta(z)
=
\begin{cases}
\displaystyle \int_0^1 \left(\frac{t^{1/2}}{(1-t)^{\beta/2}}\right)^k\,dt
=
\frac{\Gamma\!\left(1+\frac{k}{2}\right)\Gamma\!\left(1-\frac{\beta k}{2}\right)}
{\Gamma\!\left(2+\frac{k}{2}-\frac{\beta k}{2}\right)},
& \text{if } \beta k<2,\\[1.2em]
+\infty, & \text{if } \beta k\ge 2.
\end{cases}
\]
In particular, the first moment is finite iff $\beta<2$, and in that case
\[
\int_{\C} |z|\, d\mu_\beta(z)
=
\frac{\sqrt\pi}{2}\,
\frac{\Gamma\!\left(1-\frac{\beta}{2}\right)}{\Gamma\!\left(\frac{5-\beta}{2}\right)}.
\]
\end{propo}

\begin{proof}
Let $Z$ be a $\C$-valued random variable with law $\mu_\beta$, and set $R=|Z|$.
Since $F_\beta(r):=\mu_\beta(B(0,r))=\P(R\le r)$, we have for every $k\ge 0$
\[
\E[R^k]=\int_0^\infty r^k\, dF_\beta(r)=\int_0^1 \big(F_\beta^{(-1)}(t)\big)^k\,dt.
\]
Using \eqref{equ_r_t}, $F_\beta^{(-1)}(t)=t^{1/2}(1-t)^{-\beta/2}$, hence
\[
\E[R^k]=\int_0^1 t^{k/2}(1-t)^{-\beta k/2}\,dt.
\]
The integral converges near $t=1$ iff $-\beta k/2>-1$, i.e.\ $\beta k<2$; otherwise it diverges.
When $\beta k<2$, it is a Beta integral:
\[
\int_0^1 t^{k/2}(1-t)^{-\beta k/2}\,dt
=
B\!\left(1+\frac{k}{2},\,1-\frac{\beta k}{2}\right)
=
\frac{\Gamma\!\left(1+\frac{k}{2}\right)\Gamma\!\left(1-\frac{\beta k}{2}\right)}
{\Gamma\!\left(2+\frac{k}{2}-\frac{\beta k}{2}\right)}.
\]
The case $k=1$ gives the stated formula.
\end{proof}

%-----------------------------------------Example----------------------------------------------------
\subsubsection*{Example: Integer $\beta = k \ge 0$}

\begin{propo}
Let $x$ and $y$ be two $\ast$-free standard circular (``circle'') variables, let $k\in\{0,1,2,\ldots\}$, and set
\[
a = x\,y^{-k},
\]
with the convention $y^{0}=1$ (so $a=x$ when $k=0$).
Then $a$ has the stable Brown measure $\mu_\beta$ with $\beta=k$.
In particular, the case $k=0$ recovers the standard circular variable and $\mu_0$ is the uniform measure on the unit disk.
\end{propo}

\begin{proof}
If $k=0$, then $a=x$ is a standard circular variable. As is well known, $|x|^2$ is free Poisson with parameter $\lambda=1$, hence
\[
S_{|a|^2}(z)=S_{|x|^2}(z)=\frac{1}{1+z}=\frac{(-z)^0}{1+z}.
\]
Therefore $a$ has the Brown measure $\mu_0$ (the circular law), which is exactly the $\beta=0$ member of the stable family.

Assume now that $k\ge 1$. Let $\mu\in \mathrm{Prob}((0,\infty),\BB)$ and let $\mu^{-1}\equiv \mu_{\mathrm{inv}}$
be its image under $t\mapsto t^{-1}$. Then, by Proposition 3.13 in \cite{haagerup_schultz2007},
for all $z$ in a neighborhood of $(-1,0)$,
\begin{equation}
\label{S_inverted_measure}
S_{\mu^{-1}}(z)\, S_{\mu}(-1-z)=1.
\end{equation}
If $c$ is a circular random variable, then $|c|^2$ is free Poisson with parameter $\lambda=1$, hence
$S_{|c|^2}(z)=(1+z)^{-1}$. Next, by Proposition~\ref{propo_powers_R}(2) and \eqref{S_inverted_measure},
\[
S_{|c^{-1}|^2}(z)=-z.
\]
Then, applying Proposition~\ref{propo_powers_R}(1) to $x=c^{-1}$ yields
\[
S_{|c^{-k}|^2}(z)=(-z)^k.
\]
Finally, by Proposition~\ref{propo_product_R},
\[
S_{|a|^2}(z)=\frac{(-z)^k}{1+z},
\]
and therefore $a$ has the Brown measure $\mu_k$.
\end{proof}

Here is an interesting consequence. Let $x_1,\ldots,x_n,y_1,\ldots,y_n$ be $\ast$-free standard circular variables.
Then
\[
x_1 y_1^{-k} + \cdots + x_n y_n^{-k}\ \sim_B\ n^{\frac{1+k}{2}}\, x_1 y_1^{-k},
\]
where $\sim_B$ means that these two variables have the same Brown measure.

%-------------------------------Appendix---------------------------

\backmatter

\begin{appendices}

\section{Proof of \eqref{equ_S_HL} for unbounded measures}
\label{section_proof_HL_formula}\label{secA1}

\begin{lemma}
Let $\mu\in\mathrm{Prob}([0,\infty),\mathcal B)$ with $\mu\neq\delta_0$. Then there exists $\delta \in (0, \infty]$ such that the R-transform $R_\mu$ of $\mu$ is well-defined on $(-\delta, 0)$. Set $C_\mu(z):=zR_\mu(z)$. Then, there exists $\eps > 0$ such that $C_\mu(z)$ is invertible on $(-\varepsilon, 0)$ and 
\[
S_\mu(u)=\frac{1}{u}\,C_\mu^{(-1)}(u)=\frac{1}{u}\Big(uR_\mu(u)\Big)^{(-1)},
\qquad u\in(-\varepsilon,0).
\]
\end{lemma}

\begin{proof}
The Cauchy transform $G_\mu$ is well-defined, continuous and strictly decreasing on  $(-\infty, 0)$. Hence it maps $(-\infty, 0)$ to $(-\delta, 0)$ with $\delta \in (0, \infty]$.  It follows that $R_\mu(z) = G^{(-1)}(z) - \tfrac{1}{z}$ is well defined on $(-\delta, 0)$. 

\textbf{Claim:} for small $\eps_0 > 0$, the function $C_\mu(z) = z R_\mu(z) = z G^{(-1)}(z) - 1$ is strictly increasing and continuous on $(-\eps_0, 0)$, and therefore $C_\mu$  bijectively maps  $(-\eps_0, 0)$ to $(-\eps_1, 0)$ for some $\eps_1 > 0$. 

\textbf{Proof of the claim:} Define
$\psi_\mu(w)=\int_0^\infty \frac{wt}{1-wt}\,d\mu(t)$ for $w\in\mathbb C\setminus[0,\infty)$.
It is a standard property of $\psi_\mu$ that it maps $(-\infty, 0)$ to $(\mu(\{0\}) - 1, 0)$ and that it is strictly increasing and continuous on  $(-\infty, 0)$. In particular, it has a well-defined functional inverse $\chi_\mu$ on  $(\mu(\{0\}) - 1, 0)$.    Then, for $u < 0$, 
\begin{equation}
\label{equ_Cauchy_psi_basic}
1 + \psi_\mu(u) = \int_0^\infty \frac{1}{1 - ut}\, d\mu(t) = \frac{1}{u}\int_0^\infty \frac{1}{1/u - t} \, d\mu(t) = \frac{1}{u} G_\mu\Big(\frac{1}{u}\Big), 
\end{equation}
so $G_\mu\big(\frac{1}{u}\big) = u\big(1 + \psi_\mu(u)\big)$. Let $z = G_\mu\big(\frac{1}{u}\big)$. Then $u \to z$ maps $(-\infty, 0) \to (-\delta, 0)$ and the map is strictly increasing on $(-\infty, 0)$. The inverse function $u(z) = 1/G_\mu^{(-1)}(z)$ is also well-defined and strictly increasing on  $(-\delta, 0)$.

Since $R_\mu(z) = \frac{1}{u} - \frac{1}{z}$, Equation \eqref{equ_Cauchy_psi_basic} (in the form $z = u\big(1 + \psi_\mu(u)\big)$) gives 
\[
z R_\mu(z) = z\Big( \frac{1}{u} - \frac{1}{z}\Big) = \psi_\mu(u). 
\]
It follows that $z(u) R_\mu(z(u))$ is strictly increasing in $u$ and maps $(-\infty, 0) \to (\mu(\{0\}) - 1, 0)$. This implies that $z R_\mu(z)$ is strictly increasing in $z$ on interval $(-\delta, 0)$ and maps $(-\delta, 0)$ bijectively to  $(-\eps_1, 0) \subseteq (\mu(\{0\}) - 1, 0)$. This competes the proof of the claim. 

Then, writing $w=\psi_\mu(u)$, we have $u=\chi_\mu(w)$ and
$z=G_\mu(1/u)=u(1+w)=(1+w)\chi_\mu(w)$. Since $w = C_\mu(z)$, hence
\[
 C_\mu^{(-1)}(w) = z  = (1+w)\chi_\mu(w), \qquad  w \in (-\eps_1, 0),
\    
\]
and therefore
\[
S_\mu(w)=\frac{1+w}{w}\chi_\mu(w)=\frac{1}{w}\,C_\mu^{(-1)}(w).
\]

\end{proof}

%An appendix contains supplementary information that is not an essential part of the text itself but which may be helpful in providing a more comprehensive understanding of the research problem or it is information that is too cumbersome to be included in the body of the paper.
%
%%%=============================================%%
%%% For submissions to Nature Portfolio Journals %%
%%% please use the heading ``Extended Data''.   %%
%%%=============================================%%
%
%%%=============================================================%%
%%% Sample for another appendix section			       %%
%%%=============================================================%%
%
%%% \section{Example of another appendix section}\label{secA2}%
%%% Appendices may be used for helpful, supporting or essential material that would otherwise 
%%% clutter, break up or be distracting to the text. Appendices can consist of sections, figures, 
%%% tables and equations etc.
%
\end{appendices}
%
%%%===========================================================================================%%
%%% If you are submitting to one of the Nature Portfolio journals, using the eJP submission   %%
%%% system, please include the references within the manuscript file itself. You may do this  %%
%%% by copying the reference list from your .bbl file, paste it into the main manuscript .tex %%
%%% file, and delete the associated \verb+\bibliography+ commands.                            %%
%%%===========================================================================================%%

%\bibliography{comtest}% common bib file

\end{document}